\documentclass[12pt]{amsart}

\usepackage{wasysym}
\usepackage{amsmath}
\usepackage{amssymb}
\usepackage{amsthm}
\usepackage{tikz}
\usetikzlibrary{matrix,arrows,backgrounds,shapes.misc,shapes.geometric,patterns,calc,positioning}
\usetikzlibrary{calc,shapes}
\usepackage{wrapfig}
\usepackage{epsfig}  
\usepackage[all]{xy}
\usepackage{caption}
\usepackage[margin=1in]{geometry}
\usepackage{color}	 %

\input{xy}
\xyoption{poly}
\xyoption{2cell}
\xyoption{all}

\usepackage{enumitem}
\usepackage{lscape}

\usetikzlibrary{calc,shapes}
\usetikzlibrary{matrix,arrows,backgrounds,shapes.misc,shapes.geometric,patterns,calc,positioning}
\usetikzlibrary{calc,shapes}

\newcommand{\calg}{\mathcal{G}}

\makeatletter
\def\Ddots{\mathinner{\mkern1mu\raise\p@
\vbox{\kern7\p@\hbox{.}}\mkern2mu
\raise4\p@\hbox{.}\mkern2mu\raise7\p@\hbox{.}\mkern1mu}}
\makeatother

\newtheorem{thm}{Theorem}[section]
\newtheorem{prop}[thm]{Proposition}
\newtheorem{lem}[thm]{Lemma}
\newtheorem{cor}[thm]{Corollary}
\newtheorem{conjecture}[thm]{Conjecture}

\theoremstyle{definition}
\newtheorem{definition}[thm]{Definition}
\newtheorem{example}[thm]{Example}

\theoremstyle{remark}

\newtheorem{remark}[thm]{Remark}

\numberwithin{equation}{section}

\DeclareMathSizes{5}{3}{2}{2}

\newcommand\restr[2]{{
  \left.\kern-\nulldelimiterspace 
  #1 
  \vphantom{\big|} 
  \right|_{#2} 
  }}

\def\QQ{\mathbb{Q}}
\def\RR{\mathbb{R}}

\newcommand{\zd}{\delta}

\def\ZZ{\mathbb{Z}}

\newcommand{\cf}{\textup{c}} 

\definecolor{cadmiumgreen}{rgb}{0.0, 0.42, 0.24}

\begin{document}
\title{Continued Fractions and orderings on the Markov numbers}
\author{Michelle Rabideau}
\address{Department of Mathematics, University of Hartford, 
West Hartford, CT 06117-1545,USA}
\email{Rabideau@hartford.edu}
\author{Ralf Schiffler}\thanks{Both authors were supported by the NSF-CAREER grant  DMS-1254567, and by the University of Connecticut. The second author was also supported by the NSF grant DMS-1800860.}
\address{Department of Mathematics, University of Connecticut, 
Storrs, CT 06269-1009, USA}
\email{schiffler@math.uconn.edu}




%
%
%
\begin{abstract}
	
	Markov numbers are integers that appear in the solution triples of the Diophantine equation, $x^2+y^2+z^2=3xyz$, called the Markov equation.  A classical topic in number theory, these numbers are related to many areas of mathematics such as combinatorics, hyperbolic geometry, approximation theory and cluster algebras. 
	
	There is a natural map from the rational numbers between zero and one to the Markov numbers. In this paper, we prove  a  conjecture seen in Martin Aigner's book, {\em Markov's theorem and 100 years of the uniqueness conjecture}, that determine an ordering on subsets of the Markov numbers based on their corresponding rational.  The proof relies on a relationship between Markov numbers and continuant polynomials which originates in Frobenius' 1913 paper. 
\end{abstract}

 \maketitle

\setcounter{tocdepth}{1}
\tableofcontents


\section{Introduction}

 The purpose of this article is to prove a conjecture on Markov numbers from \cite{A}.


\begin{definition}
A Markov number (alternate spelling Markoff number) is any number in the triple $(x,y,z)$ of positive integer solutions to the Diophantine equation $x^2+y^2+z^2  = 3xyz$, known as the Markov equation.
\end{definition}

We consider the Markov equation rather than the more general Diophantine equation,\\ $x^2+y^2+z^2 = kxyz$, because for $k \neq 1,3$, this Diophantine equation has only the trivial solution $(0,0,0)$. Solutions to this Diophantine equation when $k=1$ are multiples of 3 times solutions to the Markov equation. Hence the Markov equation is the equation of interest. For facts about the Markov numbers we refer to \cite{A} and \cite{Re}.

The first few triples to satisfy the Markov equation are the triples containing repeated values, i.e. the singular triples $(1,1,1)$ and $(1,2,1)$. All of the other solutions are non-singular triples, some of which are depicted in Figure~\ref{markovtree}. Since the set of Markov numbers is the union of the entries in the triples, the first few Markov numbers are 1, 2, 5, 13, 29, 34, 89, 169, 194, 233, 433, 610, 985, etc.  

 Every Markov number appears as the maximum of some Markov triple. Notice that with the exception of the first non-singular triple, we only underline the maximum of each triple in the tree in Figure~\ref{markovtree}. It is known that these underlined values provide a complete list of the Markov numbers. However, it is an open conjecture by Frobenius from 1913 whether each Markov number appears as the maximum of a {\em unique} Markov triple.
 

 The Markov numbers can be indexed by the rational numbers between zero and one. This is done by comparing the combinatorially identical trees in Figure~\ref{markovtree} and Figure~\ref{fareytree}. Figure~\ref{markovtree} is the beginning of the binary tree called the Markov tree. Each branch of the tree is constructed in a specific manner. From the vertex $(x,y,z)$ the branch leading below and to the left will be $(x,3xy-z ,y)$ and below to the right will be $(y,3yz-x,z)$.

\begin{figure}
\scalebox{.8} 	{\xymatrix{ &&&(\underline{1},\underline{5},\underline{2}) \ar[lld]\ar[rrd]&&&\\
		&(1,\underline{13},5)\ar[ld]\ar[rd]&&&&(5,\underline{29},2)\ar[ld]\ar[rd]&\\
		(1,\underline{34},13)&&(13,\underline{194},5) && (5,\underline{433},29) &&(29,\underline{169},2)  }} 
	\captionof{figure}{Markov Tree (non-singular triples). The underlined values are $m_{p/q}$ where $p/q$ are values in the same position in the Farey tree.\\} \label{markovtree}
	\scalebox{1.1}{
	\scalebox{.8} {	\xymatrix{ &&&(\underline{\frac{0}{1}},\underline{\frac{1}{2}},\underline{\frac{1}{1}}) \ar[lld]\ar[rrd]&&&\\
			&(\frac{0}{1},\underline{\frac{1}{3}},\frac{1}{2})\ar[ld]\ar[rd]&&&&(\frac{1}{2},\underline{\frac{2}{3}},\frac{1}{1})\ar[ld]\ar[rd]&\\
			(\frac{0}{1},\underline{\frac{1}{4}},\frac{1}{3})&&(\frac{1}{3},\underline{\frac{2}{5}},\frac{1}{2}) && (\frac{1}{2},\underline{\frac{3}{5}},\frac{2}{3}) &&(\frac{2}{3},\underline{\frac{3}{4}},\frac{1}{1})  }}}
	\captionof{figure}{Farey Tree } \label{fareytree}
\end{figure}

In Figure~\ref{fareytree} we have the Farey tree, a binary tree of Farey triples. When starting with a triple, $\left(\frac{a}{b}, \frac{a+c}{b+d},\frac{c}{d}\right)$, we produce the next triple to the left and right respectively by $$\left(\frac{a}{b},\frac{a+(a+c)}{b+(b+d)}, \frac{a+c}{b+d}\right) \text{ and } \left( \frac{a+c}{b+d},\frac{(a+c)+c}{(b+d)+d},\frac{c}{d}\right).$$

Since the underlined values in the Farey tree provide a list of every distinct rational number from zero to one, we can correspond $\QQ_{[0,1]}$ to the Markov numbers.  We refer to a Markov number as $m_{p/q}$ where $p<q$ and $q, p$ are relatively prime positive integers.

 Therefore we are now ready to state the  fixed numerator 
 conjecture \cite[10.11]{A} that is the main topic of this paper. 

\begin{conjecture} { \it (Fixed Numerator Conjecture)\label{constant num conjecture}
	Let $p,q$ and $i$ be positive integers such that $p<q$, $\gcd(q,p) =1$ and $\gcd(q+i,p) = 1$, then $m_{p/q} < m_{p/(q+i)}$.}
\end{conjecture}
%

\begin{example}
This example highlights some of the orderings implied by Conjecture \ref{constant num conjecture}.
\[	\begin{array}{cccccccccccccccccc}
	m_{1/2} &<& m_{1/3} &<& m_{1/4} &<& m_{1/5} &<& m_{1/6} &<& \dots\\
		5 &<& 13 &<& 34 &<& 89 &<& 233 &<& \dots\\\\
	m_{2/3} &<& m_{2/5} &<& m_{2/7} &<& m_{2/9} &<& m_{2/11} &<& \dots\\
		29 &<& 194 &<& 1,325 &<& 9,077 &<& 62,210 &<& \dots\\
	\end{array}\]
\end{example}

 \begin{remark} In addition to Conjecture \ref{constant num conjecture}, Aigner presents a fixed denominator conjecture $\left( m_{p/q} < m_{(p+i)/q}\right)$ and a fixed sum conjecture $\left( m_{p/q} < m_{(p-i)/(q+i)}\right)$, which are both still open problems.
\end{remark} 

The interest of Conjecture \ref{constant num conjecture} is two-fold. On the one hand, it is a subcase of Frobenius' uniqueness conjecture and on the other hand, if the uniqueness conjecture holds, then the Markov numbers induce a total ordering $<_M$ of the rational numbers between zero and one defined by $$\dfrac{p}{q} \;<_M\; \dfrac{r}{s} \hspace{20pt} \Leftrightarrow \hspace{20pt}m_{p/q}\; <\; m_{r/s}.$$
Conjecture \ref{constant num conjecture} is a step towards understanding this total order.

 We are now ready to state our main result.
\begin{thm}
 \label{mainthm} 
The conjecture \ref{constant num conjecture} 
holds.
\end{thm}

The proof relies on a connection between Christoffel paths and Markov numbers which provides a formulation of Markov numbers as continuant polynomials. This result is already contained in Frobenius' paper, \cite{F}. We thank the referee for pointing this out. For a modern reference see Reutenauer \cite{Re}. Originally, we knew of Frobenius' result from a cluster algebra perspective which gives an alternative, albeit roundabout proof. 

 It is shown in \cite{BBH, P} that Markov triples are related to the cluster algebra of the torus with one puncture; namely, the Markov tree is obtained from the exchange graph  of the cluster algebra by specializing the initial cluster variables to 1. Then, via a formula from \cite{MSW}, one can express each Markov number as the number of perfect matchings of an associated graph, called a Markov snake graph. Finally, using results of \cite{CS4,CS5}, each Markov number can then be expressed as the numerator of a very particular continued fraction.

 Having the description of Markov numbers as continuant polynomials, the main steps in the proof of the conjecture are the following.

Theorem \ref{lem replacedifference}, which we think is interesting in its own right, is a new identity on continuant polynomials. It states that if one changes two consecutive interior entries 1,1 into one entry 2, the value of the continuant polynomial increases. 

Theorem \ref{thm 01} is a result on the continued fraction $c_{p/q}$ of a Markov number.  It states that the value of the continued fraction of the reversal of a certain initial segment of $c_{p/q}$ is smaller than the value of the continued fraction of a corresponding terminal segment. The proof of this result involves convergents of continued fractions as well as geometric arguments. 

The main result, Theorem \ref{markovordering thm}, then follows from the above and computations with continuant polynomials.


The paper is organized as follows. We start by reviewing basic properties of continued fractions in Section \ref{sect introcf} and  prove our new result on continued fractions in Section \ref{sect continued frac}.  In Section \ref{sect Markov}, we discuss the connection between Christoffel words and Markov numbers through the lense of cluster algebras which allows us to visualize continuant polynomials as snake graphs.  Section \ref{sect conjectures} is devoted to the proof our main theorem.

\section{Basic properties of continued fractions}\label{sect introcf}In this section, we provide the necessary definitions and properties of continued fractions that will be necessary to prove our results.
We restrict ourselves in this paper to finite continued fractions with non-negative integer entries. For $a_i\in \ZZ_{\geq 0}$, $a_n\neq0$,  we define $$[a_1,\dots, a_n]  := a_1 + \cfrac{1}{a_{2} + \cfrac{1}{\ddots + \cfrac{1}{a_n}}}.$$

When we evaluate a  continued fraction, we obtain a reduced rational number with numerator  we denote by $N[a_1,\dots, a_n]$, which is known as the continuant polynomial of the variables $a_1, \dots, a_n$. For example, see \cite{GKP}.   Proposition \ref{prop basicNcfprop} below lists some basic properties.

\begin{prop} \label{prop basicNcfprop}  Let $a_i\in \ZZ_{\geq 0}$. 
	\begin{equation} \label{eqnnumcutfromhead}
	N[a_1,\dots, a_n] =a_1N[a_2,\dots, a_n] +N[a_3,\dots, a_n]
	\end{equation}
	\begin{equation}\label{eqnnumcutfromtail}
	N[a_1,\dots, a_n] =a_nN[a_1,\dots, a_{n-1}] +N[a_1,\dots, a_{n-2}].
	\end{equation}
		\begin{equation}\label{eqn 112end}
	[a_1,\dots, a_n,1,1] = [a_1,\dots, a_n,2]
	\end{equation}
		\begin{equation}\label{eqn 112front}
	N[1,1,a_1,\dots, a_n] = N[2, a_1,\dots, a_n]
	\end{equation}
	\begin{equation}\label{eqnreversal}
	N[a_1,\dots, a_n] = N[a_n,\dots,a_1]
	\end{equation}
	\begin{equation}
	[a_1,\dots, a_i] < [a_1,\dots,a_n] \textup{ if $i$ is odd:}
	\end{equation}
	\begin{equation}
	[a_1,\dots, a_i] > [a_1,\dots,a_n] \textup{ if $i$ is even}.
	\end{equation}
\end{prop}

We call $a_n,\dots,a_1$ the reversal of $a_1,\dots, a_n$. We would now like to extend the  definition of the numerator of a continued fraction. First, we let $N[\hspace{5pt}] =1$, so that the recursion equations (\ref{eqnnumcutfromhead}) and (\ref{eqnnumcutfromtail}) still hold when the continued fraction has only two entries.  
\section{A Theorem on continued fractions} \label{sect continued frac}

 In this section, we prove a result on continued fractions in general. In Section \ref{sect conjectures}, we will apply it to the continued fractions related to Markov numbers.


Soon it will become tedious and unnecessary to list every entry of a continued fraction. Thus we introduce the following notation.

\begin{definition}
	Let $\mu = a_1,\dots, a_n$ be a sequence of positive integers 
	Then we define the following notation, $N[\mu] = N[a_1,\dots, a_n]$.
	
	$$\begin{array}{cclccclccccccc}
	N[\mu^-]& =& N[a_1,\dots,a_{n-1}]  &\hspace{10pt}& N[^-\mu] &=& N[a_2,\dots, a_n]     &\hspace{10pt}&  \text{for }n>1\\
	N[\mu^-] &=& N[\hspace{5pt}] =1  &\hspace{10pt}& N[^-\mu] &=&N[\hspace{5pt}] =1 &\hspace{10pt}&\text{for }n=1
	\end{array}$$
	
	$$\begin{array}{cclccclccccccc}
	N[^-\mu^-]& =& N[a_2,\dots,a_{n-1}]     &\hspace{10pt}&  \text{for }n>2\\
	N[^-\mu^-] &=& N[\hspace{5pt}] =1  &\hspace{10pt}& \text{for }n=2
	\end{array}$$
	
\end{definition}

\begin{definition}
	A \emph{replacement} is an operation on the entries of a continued fraction such that a either $1,1$ is replaced with $2$ or $2$ is replaced with 1,1.
	$$N[ \mu_1,1,1,\mu_2]  \leftrightarrow N[ \mu_1,2,\mu_2] $$
\end{definition}

%
%

	\begin{thm}\label{lem replacedifference} 	
Let each $\mu_i$ be a sequence of positive integers  
Then
	\begin{equation} \label{eqn replacedifference}
	N[\mu_1,1,1,\mu_2] - N[\mu_1,2,\mu_2] = N[\mu_1^-]N[^-\mu_2]
	\end{equation}
\end{thm}

	
\begin{proof}
	
%

We will prove the statement by induction on the number of entries before the $1,1$ in the first continued fraction.  For our base case we let $\mu_1 =a_1$. We apply Equation (\ref{eqnnumcutfromhead}) to both continuant polynomials in the expression on the left hand side of Equation (\ref{eqn replacedifference}).\\
		\begin{eqnarray*}
	N[a_1,1,1,\mu_2] - N[a_1,2,\mu_2]
		&=& a_1N[1,1,\mu_2] +N[1,\mu_2] -(a_1N[2,\mu_2]+N[\mu_2]) \\
		&=&a_1(N[1,1,\mu_2] -N[2,\mu_2]) +N[1,\mu_2]- N[\mu_2]\\
			\end{eqnarray*}
		Since $N[1,1,\mu_2] =N[2,\mu_2]$, the first term is zero. We can decompose the second term using Equation (\ref{eqnnumcutfromhead}), then combine like terms.
			\begin{eqnarray*}
	N[a_1,1,1,\mu_2] - N[a_1,2,\mu_2]	&=& N[1,\mu_2]  - N[\mu_2] \\
		&=& 1N[\mu_2] + N[^-\mu_2] - N[\mu_2] \\
		&=& N[^-\mu_2] 
	\end{eqnarray*}

Since $\mu_1 = a_1$, we have that $N[\mu_1^-] = N[\hspace{5pt}]
= 1$ 
and therefore the right side of Equation (\ref{eqn replacedifference}) is  $N[\mu_1^-]N[^-\mu_2] = N[\hspace{5pt}]N[^-\mu_2]=N[^-\mu_2] $. Therefore the statement holds in the base case. \\

Next, let $\mu_1$ have $n>1$ entries and assume Equation (\ref{eqn replacedifference}) holds for any $\mu_1$ with $n$ or less entries. We would like to prove that 
	$N[a_0,\mu_1,1,1,\mu_2] - N[a_0,\mu_1,2,\mu_2] = N[a_0,\mu_1^-]  N[^-\mu_2]$. We apply Equation (\ref{eqnnumcutfromhead}) and then regroup the expression. \\
	
	$N[a_0,\mu_1,1,1,\mu_2] - N[a_0,\mu_1,2,\mu_2]$
	\begin{eqnarray*}
		&=& a_0N[\mu_1,1,1,\mu_2] + N[^-\mu_1,1,1,\mu_2]-(a_0N[\mu_1,2,\mu_2]+N[^-\mu_1,2,\mu_2] )\\
		&=& a_0\left(N[\mu_1,1,1,\mu_2] -N[\mu_1,2,\mu_2]\right) + N[^-\mu_1,1,1,\mu_2] - N[^-\mu_1,2,\mu_2]\\
	\end{eqnarray*}
	Applying our induction hypothesis to each difference, we see that this expression is equal to
	\begin{eqnarray*}
	&& a_0N[\mu_1^-] N[^-\mu_2] + N[^-\mu_1^-] N[^-\mu_2]\\
		&=& (a_0N[\mu_1^-]  +N[^-\mu_1^-])N[^-\mu_2] \\
		&=&N[a_0,\mu_1^-]  N[^-\mu_2],\\
	\end{eqnarray*}
where the last identity holds by Equation (\ref{eqnnumcutfromhead}). Therefore Theorem \ref{lem replacedifference} is proved by induction. 
\end{proof}

\section{Markov numbers} \label{sect Markov}
In this section we study the continued fractions associated to Markov numbers. We start by  recalling the definition of Markov snake graphs and their continued fractions.

\subsection{Markov snake graphs }\label{sect markovsnake}
In this subsection, we give the background necessary to understand the relationship between Markov numbers and snake graphs. We often refer to work done in the field of cluster algebras because Markov triples are related to the clusters of the cluster algebra of a torus with one puncture \cite{BBH, P}. More specifically, Markov snake graphs, which we will define below, correspond to the cluster variables of a cluster algebra from a once punctured torus.

Let $p$ and $q$ be relatively prime integers with $p<q$. First, we define the $(q,p)$-rectangle to be the rectangle formed in $\RR^2$ with the origin and $(q,p)$ as vertices. We call the diagonal through these vertices,  $\ell_{p/q}$, because it is a line segment with slope $p/q$. The unique lattice path $L_{p/q}$ in $\ZZ \times \ZZ$ from the origin to $(q,p)$ lying strictly below $\ell_{p/q}$ and with no lattice points strictly between the path $L_{p/q}$ and the diagonal $\ell_{p/q}$ is called the Christoffel lattice path. 
For example, the Christoffel path $L_{3/5} $ is shown in Figure~\ref{ex markovsnake}.
 We construct a Markov snake graph, $\calg_{p/q}$, from the Christoffel lattice path $L_{p/q}$.

Simply put, snake graphs in general are graphs consisting of square tiles where each tile is placed either to the right or above the previous tile. General snake graphs were introduced in \cite{MSW} in order to give a combinatorial formula for cluster variables in terms of perfect matchings.  These graphs were further studied in \cite{CS,CS2,CS3,CS4,CS5,R}. The special case of Markov snake graphs already appeared in \cite{P}. The following definition is from \cite{CS5}.
\begin{definition}
	The Markov snake graph, $\calg_{p/q}$, is the snake graph with half unit length tiles, lying on the Christoffel lattice path $L_{p/q}$ such that the south west vertex of the first tile is $(0.5,0)$ and the north east vertex of the last tile is $(q, p-0.5)$.
\end{definition}

 For an example, see Figure~\ref{ex markovsnake} for the construction of the Markov snake graph $\calg_{3/5}$.  Once the Markov snake graph is constructed, we consider the number of perfect matchings it has. A perfect matching is a collection of edges in a graph such that each vertex on the graph is adjacent to exactly one edge in the collection. The following is a reformulation from \cite{CS5} of a result due to \cite{P}.\\

\begin{figure}
	\begin{center}
		\begin{minipage}{2.9in}
			\begin{center}
				\begin{tikzpicture}[scale=1.1]
				\draw[ultra thick,red, -] (0,0) -- (1,0);
				\draw[ultra thick,red, -] (1,0) -- (2,0);
				\draw[semithick, -] (2,0) -- (3,0);
				\draw[semithick, -] (3,0) -- (4,0);
				\draw[semithick, -] (4,0) -- (5,0);
				\draw[semithick, -] (0,1) -- (1,1);
				\draw[semithick, -] (1,1) -- (2,1);
				\draw[ultra thick,red, -] (2,1) -- (3,1);
				\draw[ultra thick,red, -] (3,1) -- (4,1);
				\draw[semithick, -] (4,1) -- (5,1);
				\draw[semithick, -] (0,2) -- (1,2);
				\draw[semithick, -] (1,2) -- (2,2);
				\draw[semithick, -] (2,2) -- (3,2);
				\draw[semithick, -] (3,2) -- (4,2);
				\draw[ultra thick,red, -] (4,2) -- (5,2);
				\draw[semithick, -] (0,3) -- (1,3);
				\draw[semithick, -] (1,3) -- (2,3);
				\draw[semithick, -] (2,3) -- (3,3);
				\draw[semithick, -] (3,3) -- (4,3);
				\draw[semithick, -] (4,3) -- (5,3);
				\draw[semithick, -] (0,0) -- (0,1);
				\draw[semithick, -] (0,1) -- (0,2);
				\draw[semithick, -] (0,2) -- (0,3);
				\draw[semithick, -] (1,0) -- (1,1);
				\draw[semithick, -] (1,1) -- (1,2);
				\draw[semithick, -] (1,2) -- (1,3);
				\draw[ultra thick,red, -] (2,0) -- (2,1);
				\draw[semithick, -] (2,1) -- (2,2);
				\draw[semithick, -] (2,2) -- (2,3);
				\draw[semithick, -] (3,0) -- (3,1);
				\draw[semithick, -] (3,1) -- (3,2);
				\draw[semithick, -] (3,2) -- (3,3);
				\draw[semithick, -] (4,0) -- (4,1);
				\draw[ultra thick,red, -] (4,1) -- (4,2);
				\draw[semithick, -] (4,2) -- (4,3);
				\draw[semithick, -] (5,0) -- (5,1);
				\draw[semithick, -] (5,1) -- (5,2);
				\draw[ultra thick,red, -] (5,2) -- (5,3);
				\draw[semithick, -] (0,0) -- (5,3);
				\end{tikzpicture}
			\end{center}
		\end{minipage}
		\begin{minipage}{2.9in}
			\begin{center}
				\begin{tikzpicture}[scale=1.1]
				\draw[very thick,red, -] (0,0) -- (1,0);
				\draw[very thick,red, -] (1,0) -- (2,0);
				\draw[semithick, -] (2,0) -- (3,0);
				\draw[semithick, -] (3,0) -- (4,0);
				\draw[semithick, -] (4,0) -- (5,0);
				\draw[semithick, -] (0,1) -- (1,1);
				\draw[semithick, -] (1,1) -- (2,1);
				\draw[ultra thick,red, -] (2,1) -- (3,1);
				\draw[ultra thick,red, -] (3,1) -- (4,1);
				\draw[semithick, -] (4,1) -- (5,1);
				\draw[semithick, -] (0,2) -- (1,2);
				\draw[semithick, -] (1,2) -- (2,2);
				\draw[semithick, -] (2,2) -- (3,2);
				\draw[semithick, -] (3,2) -- (4,2);
				\draw[ultra thick,red, -] (4,2) -- (5,2);
				\draw[semithick, -] (0,3) -- (1,3);
				\draw[semithick, -] (1,3) -- (2,3);
				\draw[semithick, -] (2,3) -- (3,3);
				\draw[semithick, -] (3,3) -- (4,3);
				\draw[semithick, -] (4,3) -- (5,3);
				\draw[semithick, -] (0,0) -- (0,1);
				\draw[semithick, -] (0,1) -- (0,2);
				\draw[semithick, -] (0,2) -- (0,3);
				\draw[semithick, -] (1,0) -- (1,1);
				\draw[semithick, -] (1,1) -- (1,2);
				\draw[semithick, -] (1,2) -- (1,3);
				\draw[ultra thick,red, -] (2,0) -- (2,1);
				\draw[semithick, -] (2,1) -- (2,2);
				\draw[semithick, -] (2,2) -- (2,3);
				\draw[semithick, -] (3,0) -- (3,1);
				\draw[semithick, -] (3,1) -- (3,2);
				\draw[semithick, -] (3,2) -- (3,3);
				\draw[semithick, -] (4,0) -- (4,1);
				\draw[ultra thick,red, -] (4,1) -- (4,2);
				\draw[semithick, -] (4,2) -- (4,3);
				\draw[semithick, -] (5,0) -- (5,1);
				\draw[semithick, -] (5,1) -- (5,2);
				\draw[ultra thick,red, -] (5,2) -- (5,3);
				\draw[semithick,blue, -] (.5,0.02) -- (2,0.02);
				\draw[semithick,blue, -] (.5,.52) -- (2,.52);
				\draw[semithick,blue, -] (1.5,1.02) -- (4,1.02);
				\draw[semithick,blue, -] (1.5,1.52) -- (4,1.52);
				\draw[semithick,blue, -] (3.5,2.02) -- (5,2.02);
				\draw[semithick,blue, -] (3.5,2.52) -- (5,2.52);
				\draw[semithick,blue, -] (.5,0) -- (.5,.5);
				\draw[semithick,blue, -] (1,0) -- (1,.5);
				\draw[semithick,blue, -] (1.5,0) -- (1.5,1.5);
				\draw[semithick,blue, -] (1.98,0) -- (1.98,1);
				\draw[semithick,blue, -] (2,1) -- (2,1.5);
				\draw[semithick,blue, -] (2.5,1) -- (2.5,1.5);
				\draw[semithick,blue, -] (3,1) -- (3,1.5);
				\draw[semithick,blue, -] (3.5,1) -- (3.5,2.5);
				\draw[semithick,blue, -] (3.98,1) -- (3.98,2);
				\draw[semithick,blue, -] (4,2) -- (4,2.5);
				\draw[semithick,blue, -] (4.5,2) -- (4.5,2.5);
				\draw[semithick,blue, -] (4.98,2) -- (4.98,2.5);
				\end{tikzpicture}
			\end{center}
		\end{minipage}
	\end{center}  
	\captionof{figure}{On the left we have the line $\ell_{3/5}$ with the unique Christoffel lattice path $L_{3/5}$ marked in red. On the right we have the Markov snake graph $\calg_{3/5}$ in blue lying on the Christoffel lattice path. The continued fraction $\cf_{p/q}$ associated to $\calg_{3/5}$ is $[2,2,2,1,1,2,2,2]$, which means that $\calg_{3/5}$ has $N[2,2,2,1,1,2,2,2] = 433$ perfect matchings. \\}\label{ex markovsnake}
\end{figure}

\begin{thm}
	Let $q,p$ be relatively prime positive integers. The number of perfect matchings of the Markov snake graph, $\calg_{p/q}$, in the $(q,p)$-rectangle is the Markov number $m_{p/q}$.  \label{thm pm= Markov}
\end{thm}

According to Theorem \ref{thm pm= Markov}, 
 when considering Conjecture \ref{constant num conjecture}, we can instead analyze the number of perfect matchings of the corresponding Markov snake graphs.
 
 When $p$ and $q$ are not relatively prime, we can still associate a numerical value to $m_{p/q}$ in a somewhat analogous manner. We construct a unique lattice path $L_{p/q}$ in the $(q,p)$-rectangle from the origin to $(q,p)$ such that $L_{p/q}$ lies below or on the line segment $\ell_{p/q}$ from the origin to $(q,p)$ and no lattice points lie strictly between $L_{p/q}$ and $\ell_{p/q}$. Then we construct a snake graph on this lattice path in the same manner as before. We call this a lattice path snake graph, rather than the more specific Markov snake graph. We let the number of perfect matchings of this lattice path snake graph be $m_{p/q}$. 

\subsection{Continued fractions of Markov snake graphs}
Every snake graph has a corresponding continued fraction. Moreover, the numerator of that continued fraction is the number of perfect matchings of its associated snake graph. This relation to continued fractions was found in \cite{CS4} and applications were given in \cite{CS5, CLS, LS,R}. Therefore by Theorem~\ref{thm pm= Markov}, the numerator of the continued fraction associated to a Markov snake graph is that Markov number. Thus we will study Markov numbers by analyzing the numerators of continued fractions, or, the continuant polynomials.


 Regardless of whether $p$ and $q$ are relatively prime, we study $m_{p/q}$ by considering its associated continued fraction which we denote by $\cf_{p/q}$. The continued fraction of a snake graph is determined by the snake graph's sign function as in \cite{CS4}. For lattice path snake graphs, including Markov snake graphs, we can determine the entries in the continued fraction by the following process. Shade the first and last tiles in the snake graph, then shade any corner tiles. The entries in the continued fraction can then be read off the snake graph. Any shaded tile represents an entry 2 and each interior edge strictly between shaded tiles represents an entry 1. See Example \ref{ex shadedtiles}.
   \begin{example}\label{ex shadedtiles} Here we have the Markov snake graph associated to the $(5,3)$-rectangle. The continued fraction is $\cf_{3/5}=[2,2,2,1,1,2,2,2]$, which has numerator equal to $m_{3/5} = 433$.
 	\begin{center}
 		\begin{tikzpicture}[scale=1.1]
 		\draw[semithick,blue, -] (.5,0.02) -- (2,0.02);
 		\draw[semithick,blue, -] (.5,.52) -- (2,.52);
 		\draw[semithick,blue, -] (1.5,1.02) -- (4,1.02);
 		\draw[semithick,blue, -] (1.5,1.52) -- (4,1.52);
 		\draw[semithick,blue, -] (3.5,2.02) -- (5,2.02);
 		\draw[semithick,blue, -] (3.5,2.52) -- (5,2.52);
 		\draw[semithick,blue, -] (.5,0) -- (.5,.5);
 		\draw[semithick,blue, -] (1,0) -- (1,.5);
 		\draw[semithick,blue, -] (1.5,0) -- (1.5,1.5);
 		\draw[semithick,blue, -] (1.98,0) -- (1.98,1);
 		\draw[semithick,blue, -] (2,1) -- (2,1.5);
 		\draw[semithick,blue, -] (2.5,1) -- (2.5,1.5);
 		\draw[semithick,blue, -] (3,1) -- (3,1.5);
 		\draw[semithick,blue, -] (3.5,1) -- (3.5,2.5);
 		\draw[semithick,blue, -] (3.98,1) -- (3.98,2);
 		\draw[semithick,blue, -] (4,2) -- (4,2.5);
 		\draw[semithick,blue, -] (4.5,2) -- (4.5,2.5);
 		\draw[semithick,blue, -] (4.98,2) -- (4.98,2.5);
 		\fill [blue] (0.5,0) rectangle (1,0.5);
 		\fill [blue] (1.5,0) rectangle (2,0.5);
 		\fill [blue] (1.5,1) rectangle (2,1.5);
 		\fill [blue] (3.5,1) rectangle (4,1.5);
 		\fill [blue] (3.5,2) rectangle (4,2.5);
 		\fill [blue] (4.5,2) rectangle (5,2.5);
 		\node at (0.7,-0.3) {$2$};
 		\node at (1.7,-0.3) {$2$};
 		\node at (2.2,.7) {$2$};
 		\node at (2.6,0.7) {$1$};
 		\node at (3.1,0.7) {$1$};
 		\node at (3.7,0.7) {$2$};
 		\node at (4.2,1.7) {$2$};
 		\node at (4.7,1.7) {$2$};
 		\end{tikzpicture}
 	\end{center}
 \end{example}

\begin{remark}
\label{rem 1}  
 The continued fraction, $[a_1, \dots, a_n]$ has entries $a_i \in \{1,2\}$ such that the sum of the entries is $\sum_{i=1} ^n a_i = 2q+2p-2$ . Moreover,  
 $a_1=a_n=2$, $n$ is even, and $a_i=a_{i+1}$ whenever $i$ is even and $2\le i\le n-2$. 
 \end{remark}
 
  While we are going to consider each 2 separately,  we would like  to refer to the pairs of 1's as a single entry.

\begin{definition} Let $[a_1,a_2,\dots,a_n]$ be the continued fraction associated to a lattice path snake graph. 
	The sequence $a_1,a_2,...,a_n$ can be decomposed into the subsequence $\nu_1,...,\nu_m$ where each $\nu_i =1,1$ or $\nu_i =2$ such that we have an identity of sequences  $a_1,a_2,...,a_n = \nu_1,\nu_2,...,\nu_m$. Then each $\nu_i$ is called a \emph{replaceable entry}. 
\end{definition}

\begin{example}\label{m(7,3)andm(7,4)}
	In this example, we observe that $m_{p/q} < m_{p/(q+1)}$ when $q= 7$ and $p = 3$.  In the continuant polynomial below, we have highlighted the different replacements. Notice that the $(q+1,p)$-Markov continuant polynomial has one more replaceable entry than the $(q,p)$-Markov continuant polynomial. \\ 
	
$	\begin{array}{rccccccccccl}
	
	m_{p/(q+1)}=m_{3/8}=N[&2, 1,1,2,2,1,1,&\color{red} 1,1,&2,& \color{cadmiumgreen}2,&\color{red}1,1,&2]& = 7,561
\\
m_{p/q} = m_{3/7} = N[ & 2, 1,1,  2,2,1,1, &\color{red} 2,&2, &  \color{cadmiumgreen}1,1,& \color{red}2&]& =2,897\\	\end{array}$\\
\end{example}


It will be of interest for us to know the number of replaceable entries in the continued fraction associated to $m_{p/q}$. The sum of the entries in the continued fraction is the sum of twice the number of pairs of 1's with twice the number of 2's that appear. Hence, we have $2q+2p-2 = 2( \# \text{ of pairs of 1's}) + 2(2p)$, which implies that the number of replaceable entries in the continued fraction associated to $m_{p/q}$  is $q+p-1$. In Example~\ref{ex twosnakes} we compare the two Markov snake graphs for $m_{15/22}$ and $m_{15/23}$ to see how the associated continued fractions differ by replacements.

\begin{example} \label{ex twosnakes} In this example, the Markov snake graph $\calg_{15/23}$ is shown in blue on the same graph as $\calg_{15/22}$ in red, with their overlap in purple. The black shaded tiles represent the tiles for which a replacement in the corresponding continued fraction occurs. $m_{15/23} =187,611,224,490,881$ and $m_{15/22}=73,224,462,646,361$\\
	
	\noindent
	\scalebox{.75}{
		$\arraycolsep=1pt
		\begin{array}{rcccccccccccccccccccccccccll} 
		\color{red}	m_{15/22}=N
[\color{red}2,2,2,&2,&\color{red}2,&1,1,&\color{red}2,2,&2,&\color{red}2,&1,1,&\color{red}2,2,&2,&\color{red}2,&1,1,&\color{red}2,2,&2,&\color{red}2,&1,1,&\color{red}2,2,&2,&\color{red}2,&1,1,&\color{red}2,2,&2,&\color{red}2,&1,1,&\color{red}2,2,&2,&\color{red}2,2]\\
\color{blue}	m_{15/23}=N[ 2,2,2,&1,1,&\color{blue}2,&2,&\color{blue}2,2,&1,1,&\color{blue}2,&2,&\color{blue}2,2,&1,1,&\color{blue}2,&2,&\color{blue}2,2,&1,1,&\color{blue}2,&2,&\color{blue}2,2,&1,1,&\color{blue}2,&2,&\color{blue}2,2,&1,1,&\color{blue}2,&2,&\color{blue}2,2,&1,1,&\color{blue}2,2,2]\\
		\end{array}$} 
	\begin{center}
	\begin{tikzpicture}[scale=.53]
	\node at (.3,-0.5) {$(0,0)$};
	\node [blue] at (24.3,14.8) {$(23,15)$};
	\node [red] at (23,15.8) {$(22,15)$};
	\draw[semithick, -](0,0) rectangle (23,15);
	\draw[step=1.0,black,thin] (0,0) grid (23,15);
	\draw[semithick,blue, -](0,0)-- (23,15);
	\draw[semithick,red, -](0,0)-- (22,15);
	\fill [blue,opacity=0.5]  (0.5,0) rectangle (2,.5);
	\fill [blue,opacity=0.5]  (1.5,.5) rectangle (2,1.5);
	\fill [blue,opacity=0.5]  (2,1) rectangle (4,1.5);
	\fill [blue,opacity=0.5]  (3.5,1.5) rectangle (4,2.5);
	\fill [blue,opacity=0.5]  (4,2) rectangle (5,2.5);
	\fill [blue,opacity=0.5] (4.5,2.5) rectangle (5,3.5);
	\fill [blue,opacity=0.5]  (5,3) rectangle (7,3.5);
	\fill [blue,opacity=0.5]  (6.5,3.5) rectangle (7,4.5);
	\fill [blue,opacity=0.5] (7,4) rectangle (8,4.5);
	\fill [blue,opacity=0.5] (7.5,4.5) rectangle (8,5.5);
	\fill [blue,opacity=0.5]  (8,5) rectangle (10,5.5);
	\fill [blue,opacity=0.5] (9.5,5.5) rectangle (10,6.5);
	\fill [blue,opacity=0.5] (10,6) rectangle (11,6.5);
	\fill [blue,opacity=0.5] (10.5,6.5) rectangle (11,7.5);
	\fill [blue,opacity=0.5]  (11,7) rectangle (13,7.5);
	\fill [blue,opacity=0.5]  (12.5,7.5) rectangle (13,8.5);
	\fill [blue,opacity=0.5] (13,8) rectangle (14,8.5);
	\fill [blue,opacity=0.5]  (13.5,8.5) rectangle (14,9.5);
	\fill [blue,opacity=0.5]  (14,9) rectangle (16,9.5);
	\fill [blue,opacity=0.5]  (15.5,9.5) rectangle (16,10.5);
	\fill [blue,opacity=0.5]  (16,10) rectangle (17,10.5);
	\fill [blue,opacity=0.5]  (16.5,10.5) rectangle (17,11.5);
	\fill [blue,opacity=0.5]  (17,11) rectangle (19,11.5);
	\fill [blue,opacity=0.5]  (18.5,11.5) rectangle (19,12.5);
	\fill [blue,opacity=0.5]  (19,12) rectangle (20,12.5);
	\fill [blue,opacity=0.5] (19.5,12.5) rectangle (20,13.5);
	\fill [blue,opacity=0.5]  (20,13) rectangle (22,13.5);
	\fill [blue,opacity=0.5]  (21.5,13.5) rectangle (22,14.5);
	\fill [blue,opacity=0.5] (22,14) rectangle (23,14.5);
	\fill [red,opacity=0.5] (0.5,0) rectangle (2,.5);
	\fill [red,opacity=0.5] (1.5,.5) rectangle (2,1.5);
	\fill [red,opacity=0.5] (2,1) rectangle (3,1.5);
	\fill [red,opacity=0.5] (2.5,1.5) rectangle (3,2.5);
	\fill [red,opacity=0.5] (3,2) rectangle (5,2.5);
	\fill [red,opacity=0.5] (4.5,2.5) rectangle (5,3.5);
	\fill [red,opacity=0.5] (5,3) rectangle (6,3.5);
	\fill [red,opacity=0.5] (5.5,3.5) rectangle (6,4.5);
	\fill [red,opacity=0.5] (6,4) rectangle (8,4.5);
	\fill [red,opacity=0.5] (7.5,4.5) rectangle (8,5.5);
	\fill [red,opacity=0.5] (8,5) rectangle (9,5.5);
	\fill [red,opacity=0.5] (8.5,5.5) rectangle (9,6.5);
	\fill [red,opacity=0.5] (9,6) rectangle (11,6.5);
	\fill [red,opacity=0.5] (10.5,6.5) rectangle (11,7.5);
	\fill [red,opacity=0.5] (11,7) rectangle (12,7.5);
	\fill [red,opacity=0.5] (11.5,7.5) rectangle (12,8.5);
	\fill [red,opacity=0.5] (12,8) rectangle (14,8.5);
	\fill [red,opacity=0.5] (13.5,8.5) rectangle (14,9.5);
	\fill [red,opacity=0.5] (14,9) rectangle (15,9.5);
	\fill [red,opacity=0.5] (14.5,9.5) rectangle (15,10.5);
	\fill [red,opacity=0.5] (15,10) rectangle (17,10.5);
	\fill [red,opacity=0.5] (16.5,10.5) rectangle (17,11.5);
	\fill [red,opacity=0.5] (17,11) rectangle (18,11.5);
	\fill [red,opacity=0.5] (17.5,11.5) rectangle (18,12.5);
	\fill [red,opacity=0.5] (18,12) rectangle (20,12.5);
	\fill [red,opacity=0.5] (19.5,12.5) rectangle (20,13.5);
	\fill [red,opacity=0.5] (20,13) rectangle (21,13.5);
	\fill [red,opacity=0.5] (20.5,13.5) rectangle (21,14.5);
	\fill [red,opacity=0.5] (21,14) rectangle (22,14.5);
	\fill (2.5,1) rectangle (3, 1.5);
	\fill (3.5,2) rectangle (4, 2.5);
	\fill (5.5,3) rectangle (6, 3.5);
	\fill (6.5,4) rectangle (7, 4.5);
	\fill (8.5,5) rectangle (9, 5.5);
	\fill (9.5,6) rectangle (10, 6.5);
	\fill (11.5,7) rectangle (12, 7.5);
	\fill (12.5,8) rectangle (13, 8.5);
	\fill (14.5,9) rectangle (15, 9.5);	
	\fill (15.5,10) rectangle (16, 10.5);	
	\fill (17.5,11) rectangle (18, 11.5);
	\fill (18.5,12) rectangle (19, 12.5);
	\fill (20.5,13) rectangle (21, 13.5);
	\end{tikzpicture}
	\end{center}
\end{example}

\subsection{A theorem on the continued fraction of a Markov snake graph}
	
Let $m_{p/q}$ be the Markov number of slope $p/q<1$ and denote by $\calg_{p/q}$ the corresponding Markov snake graph and by $\cf_{p/q}$ its continued fraction. Thus $m_{p/q}$ is the numerator of $\cf_{p/q}$. Recall that $\cf_{p/q}$ has even length and that its first and its last entry is a 2.

\begin{thm} \label{thm 01}  
	Let $\cf_{p/q}=[\mu, 2, \delta, 1,1,\nu]$, where $\mu,\delta $ and $\nu$ all have odd length and $\delta $ is a sequence of $2$s. Let $\tilde{\mu}$ be the sequence $\mu$ in reverse order.
 Then we have the following inequality of continued fractions
  \[[\tilde{\mu}] <[\nu].\]
	\end{thm}

\begin{proof}
Let $\mu=a_\ell,a_{\ell-1},\ldots, a_2,a_1$ so $\tilde {\mu}=a_1,a_2,\ldots,a_\ell$. By  Remark \ref{rem 1}, we have  $a_\ell=2$, $a_{\ell-1}=a_{\ell-2},\ldots,a_4=a_3, a_2=a_1$.
Thus 
\begin{equation}\label{eq 1} a_i=a_{i+1}, \textup{ for all odd $i$, $i=1,3,\ldots,\ell-2$}.\end{equation}

 Assume that $\tilde{\mu}$ and $\nu$ agree 
up to but not including position $k$. In other words, the first difference in the sequence occurs at position $k$.\\
	
{\bf Case A:} $\tilde{\mu} = a_1,  a_2 \dots, a_{k-1},a_k,\dots, a_{\ell} \text{ and } \nu = a_1, \dots, a_{k-1}$.\vspace{10pt}\\
Then $a_{k-1}$ is the last entry of $\cf_{p/q}$ and hence $a_{k-1}= 2$. Moreover, $k-1$ is odd, since $\nu$ has odd length, and thus $a_k=2$ by (\ref{eq 1}).
We shall show that this case is impossible. Consider the Markov snake graph $\calg_{p/q}$ corresponding to the continued fraction $[\mu,2, \delta,1,1,\nu]$ in the $(q,p)$-rectangle, see Figure \ref{fig 1}. Denote the lattice point on the south-east vertex of the $a_k$ tile by $(x_1,y_1)$ and denote the lattice point on the south-east vertex of the last tile in $\delta$ by $(x_2,y_2)$. Then we let $(x_3,y_3)$ be the lattice point one unit directly above the south-east vertex of the $a_{k-1}$ tile in $\nu$.  
By construction, $(x_3,y_3)=(q,p)$. 
 Let $\calg(\mu)$ and $\calg(\nu)$ be the subgraphs corresponding to the segments $\mu$ and $\nu$, see Figure \ref{fig 1}.

\begin{figure}
\begin{center}
\scalebox{.8}{ 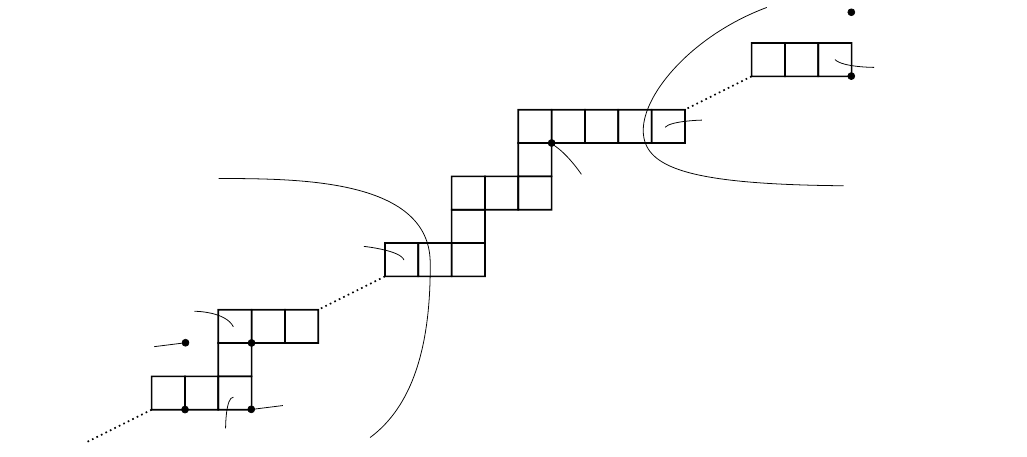}
 \caption{Proof of Theorem \ref{thm 01} Case A; $k$ is even. }\label{fig 1}
\end{center}
\end{figure}

To obtain a contradiction, we shall show that $(x_3,y_3)$ lies strictly below the line $\ell_{p/q}$. Equivalently, we shall show $\frac{y_3}{x_3} < \frac{p}{q}$. Since $(x_2,y_2)$ lies on the Christoffel lattice path, we know that $\frac{y_2}{x_2} < \frac{p}{q}$, hence
it suffices to show 
  $\frac{y_3}{x_3} <\frac{y_2}{x_2}$.

We introduce the following notation.
\[\begin{array}{rcl}
a&=& \frac{1}{2}\,\#\{i\in\{1,2,\ldots,k-2\}\mid a_i=2\}\\[5pt]
b&=& a+ \frac{1}{2}\,\#\{i\in\{1,2,\ldots,k-2\}\mid a_i=1\}\\[5pt]
c&=& \frac{1}{2}\,(\textup{number of twos between $\mu$ and $\nu$ in $\cf_{p/q}$}) = \frac{1}{2}(1+ \textup{(number of twos in $\delta$)}).
\end{array}\]
Thus $a$ is the number of pairs of twos in $\tilde{\mu}$ up to (not including) $a_{k-1}$,  $b$ is the sum of the number of pairs of ones and pairs of twos in $\tilde{\mu}$ up to (not including) $a_{k-1}$.
Then $x_2-x_1 = b+c $ and $y_2-y_1 = a+c+1 $, where the extra 1 accounts for the  vertical step at $a_{k-1}$. Similarly $x_3-x_1=2b+c+2$, where the extra 2 accounts for the  horizontal segment of length 3/2 between $(x_2,y_2)$  and $\calg(\nu)$ 
and the last tile labeled $a_{k-1}$ of width 1/2. Also 
 $y_3-y_1=2a+c+2$.
We thus have the following coordinates.
\begin{equation}\label{eq 2}(x_2,y_2) = ( x_1 + b+c, y_1 +a+c+1) \hspace{10pt} \text{and} \hspace{10pt} 
(x_3,y_3) = (x_1+ 2b+c+2, y_1+2a+c+2).\end{equation}

Therefore, in order to show that $\frac{y_3}{x_3} <\frac{y_2}{x_2}$, it suffices to show
\[(x_1 + b+c)(y_1+2a+c+2) < (x_1+ 2b+c+2)(y_1 +a+c+1).\]
Expanding and collecting like terms shows that this inequality is equivalent to 
\[\dfrac{a+1}{b+2} < \dfrac{y_1 +c}{x_1+c-2}.\]
%

Now, the point $(x_1 +c-2, y_1+c)= (x_1-1 +(c-1), y_1+1+(c-1))$ lies above the line $\ell_{p/q}$, because the point $(x_1-1,y_1+1)$ does, see Figure \ref{fig 1}. Therefore $\frac{p}{q} <\frac{y_1 +c}{x_1+c-2}$. Thus it suffices to show that $	\frac{a+1}{b+2} < \frac{p}{q}$.
Using the coordinates (\ref{eq 2}), we see that 
$$\frac{a+1}{b+2} 
= \frac{ y_2 -( y_1+c)}{x_2 - (x_1+c-2)}. $$
Therefore $\frac{a+1}{b+2} $ is the slope of the line segment from $(x_1 +c-2, y_1+c)$ to $(x_2,y_2)$, and since $(x_1 +c-2, y_1+c)$ lies above the line $\ell_{p/q}$ and to the left of $(x_2,y_2)$, and $(x_2,y_2)$ lies below the line $\ell_{p/q}$, we have $\frac{a+1}{b+2} < \frac{p}{q}$. This completes the proof in case A.\\

{\bf Case B:} $\tilde{\mu} = a_1, \dots, a_{k-1} \text{ and } \nu = a_1, \dots, a_{k-1}, b_k,\dots, b_{k+n}$ \vspace{10pt}\\
Since $\tilde{\mu}$ has odd length,  $k-1$ is odd. 
In this case, $[\tilde{\mu}]$ is the $k-1$ convergent of $[\nu]$. Therefore   $[\tilde{\mu}]< [\nu]$ because odd convergents of a continued fraction increase towards the limit of the continued fraction, see Proposition \ref{prop basicNcfprop}.\vspace{10pt}\\

{\bf Case C:} $\tilde{\mu} = a_1, \dots, a_{k-1},a_k, \dots, a_{\ell} \text{ and } \nu = a_1, \dots, a_{k-1}, b_k,\dots, b_{m}$ where $a_k \neq b_k$.	\vspace{10pt}\\
First we show that $k$ is odd. Note that the first entry of $\nu$ sits at an even position in the continued fraction $\cf_{p/q}=[\mu,2,\delta,1,1,\nu]$, since $\mu$ and $\delta $ both have odd length. Thus, if $k$ was even, then the $k$-th entry $b_k$ in $\nu$ sits at an odd position in the continued fraction, and by Remark \ref{rem 1} this implies that it is equal to the preceding entry, thus $a_{k-1}=b_k$. On the other hand, (\ref{eq 1}) implies  $a_k=a_{k-1}$ and  we get
$a_k=b_k$, a contradiction.
Thus
 $k$ must be odd.
 
 We distinguish two subcases.
  Suppose first that $a_k =1$. Then $a_{k+1} =1$, and since $a_k \neq b_k$, $b_k =2$. Then $[\nu]$ is larger than its odd convergent $[a_1, \dots, a_{k-1},2]$ and the even convergent, $[a_1, \dots, a_{k-1},1,1]$, of $[\tilde{\mu}]$ is larger than $[\tilde{\mu}]$. Therefore by transitivity, $[\nu] > [\tilde{\mu}]$.
  
Now suppose that  $a_k =2$. We shall show that this case is impossible. Consider the Markov snake graph corresponding to the continued fraction $[\mu,2, \delta,1,1,\nu]$ in the $(q,p)$-rectangle, see Figure \ref{fig 2}. As in case A, we denote the lattice point on the south-east vertex of the $a_k$ tile by $(x_1,y_1)$, note however that now, since $k$ is odd, this tile corresponds  to an upper left corner in the snake graph. Also as in case A, we denote  the lattice point on the south-east vertex of the last tile in $\delta$ by $(x_2,y_2)$. But now we let $(x_3,y_3)$ be the lattice point one unit directly above the south vertex of the $b_{k+1}=1$ edge. 

\begin{figure}
\begin{center}
\scalebox{.8}{ 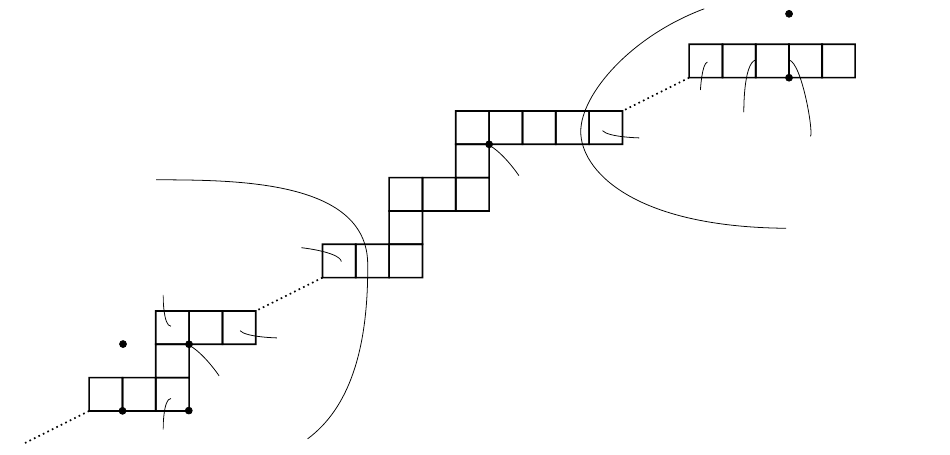} 
 \caption{Proof of Theorem \ref{thm 01} Case C; $k$ is odd.}\label{fig 2}
\end{center}
\end{figure}

 Since the lattice point $(x_3,y_3)$ lies above the snake graph, it also  lies above the Christoffel lattice path. To obtain a contradiction, we shall show that $(x_3,y_3)$ lies below the line $\ell_{p/q}$. Equivalently, we shall show $\frac{y_3}{x_3} < \frac{p}{q}$. Since $(x_2,y_2)$ lies on the Christoffel lattice path, we know that $\frac{y_2}{x_2} < \frac{p}{q}$, hence it suffices to show $\frac{y_3}{x_3} <\frac{y_2}{x_2}$.

Similar to case A, 
we introduce the following notation.
\[\begin{array}{rcl}
a&=& \frac{1}{2}\,\#\{i\in\{1,2,\ldots,k-1\}\mid a_i=2\}\\[5pt]
b&=& a+ \frac{1}{2}\,\#\{i\in\{1,2,\ldots,k-1\}\mid a_i=1\}\\[5pt]
c&=& \frac{1}{2}\,(\textup{number of twos between $\mu$ and $\nu$ in $\cf_{p/q}$})=\frac{1}{2}(1+ \textup{(number of twos in $\delta$)}).
\end{array}\]
The difference with case A is that in the definition of $a$ and $b$, the index $i$ is allowed to be $k-1$.

Then $x_2-x_1 = b+c $ and $y_2-y_1 = a+c $, because now the point $(x_1,y_1)$ is a vertex in an upper left corner tile.  Similarly $x_3-x_1=2b+c+2$, where the extra 2 again accounts for the horizontal segment of length 3/2 between $(x_2,y_2)$  and $\calg(\nu)$ 
and the  tile whose vertical edges are labeled $b_k,b_{k-1}$ of width 1/2. Also $y_3-y_1=2a+c+1$.
We thus have the following coordinates.

\begin{equation}
\label{eq 3}(x_2,y_2) = ( x_1 + b+c, y_1 +a+c) \hspace{10pt} \text{and} \hspace{10pt} 
(x_3,y_3) = (x_1+ 2b+c+2, y_1+2a+c+1). 
\end{equation}

Therefore, in order to show that $\frac{y_3}{x_3} <\frac{y_2}{x_2}$, it suffices to show
\[(x_1 + b+c)(y_1+2a+c+1) < (x_1+ 2b+c+2)(y_1 +a+c).\]

Expanding and collecting like terms shows that this inequality is equivalent to 
\[			\dfrac{a+1}{b+2} < \dfrac{y_1 +c-1}{x_1+c-2}.\]

Now, the point $(x_1 +c-2, y_1+c-1)$ lies above the line $\ell_{p/q}$, because $(x_1-1,y_1)$ does, see Figure \ref{fig 2}. Therefore $\frac{p}{q} <\frac{y_1 +c-1}{x_1+c-2}$, and hence it suffices to show that $	\frac{a+1}{b+2} < \frac{p}{q}$.
Using the coordinates (\ref{eq 3}), we see that
$$\dfrac{a+1}{b+2}= \dfrac{ y_2 -( y_1+c-1)}{x_2 - (x_1+c-2)}. $$

Therefore $\frac{a+1}{b+2} $ is the slope of the line segment from $(x_1 +c-2, y_1+c-1)$ to $(x_2,y_2)$, and  since $(x_1 +c-2, y_1+c-1)$ is above the line $\ell_{p/q}$ and to the left of $(x_2,y_2)$, and $(x_2,y_2)$ lies below the line $\ell_{p/q}$ (since it is on the Christoffel path), we have $\frac{a+1}{b+2} < \frac{p}{q}$. This completes the proof.
	\end{proof}

The theorem has the following immediate consequence on the continuant polynomials.

\begin{cor} \label{cor 01} With the notation of Theorem \ref{thm 01}
	\[N[\mu] N[^-\nu]<N[\mu^-]N[\nu].\]
	\end{cor}

\begin{proof}
The theorem yields $[\tilde{\mu}]< [\nu]$, which means $\cfrac{N[\tilde\mu]}{N[^-\tilde\mu]}<\cfrac{N[\nu]}{N[^-\nu]}$.
The result now follows, because $ N[ \tilde{\mu}] =N[\mu]$ and  $ N[^- \tilde{\mu}] =N[\mu^-]$,  by  equation (\ref{eqnreversal}).
	\end{proof}

%
%
%

\section{Proof of Conjecture \ref{constant num conjecture}} \label{sect conjectures}%

The purpose of this section is to provide a proof for the fixed numerator conjecture seen in \cite{A} and reworded in Conjecture~\ref{constant num conjecture}. In order to do so, we will prove a more general statement, Theorem~\ref{markovordering thm}. 

\subsection{A preparatory lemma} We start with a lemma on the difference between the continued fractions $\cf_{p/q}$ and $\cf_{p/(q+1)}$.
\begin{lem}
 \label{lem 5}
 Let $p$ and $q$ be positive integers such that $p<q$. 
  
 (a) If the lattice point $(a,b)$ with $a<p, b<q$ lies between the two line segments $\ell_{p/q}$ and $\ell_{p/q+1}$ then none of the points $(a-1,b),(a+1,b), (a,b-1), (a,b+1)$ lies between the line segments.
 
 (b) If $\cf_{p/q}=[\mu,2,\delta,1,1,\nu]$ and $\cf_{p/q+1}=[\mu',1,1,\delta,2,\nu']$, where $\mu,\mu'$ have the same odd length, then all entries of the sequence $\delta $  are equal to 2.
\end{lem}

\begin{proof}
(a) The maximal horizontal distance between the line segments $\ell_{p/q}$ and $\ell_{p/(q+1)}$ is 1, and it is attained at the endpoints $(q,p), (q+1,p)$. The maximal vertical distance is attained at the points $(q,p)$ and $(q,pq/(q+1))$, and it is equal to $p/(q+1)<1$. This proves (a).

(b) By definition, the Christoffel paths of $(q,p) $ and $(q+1,p)$ go through the same lattice points except for those points  that lie between the line segments $\ell_{p/q}$ and $\ell_{p/q+1}$. Suppose $(a,b)$ is a lattice point that lies between the line segments. 
Then by part (a) we know that both Christoffel paths go through the points $(a,b-1) $ and $(a+1,b)$. We distinguish two cases. 

If the point $(a+1,b+1)$ does not lie between the lines then we have the situation illustrated in  the left picture  of Figure \ref{fig 3}, where both paths go through the point $(a+2,b)$. The corresponding snake graphs are indicated in red or blue and the common tiles in purple. The continued fraction $\cf_{p/q+1}$ is $[\ldots, 1,1,2,2,\ldots]$ where the 1,1 corresponds to the tile whose southeast vertex is $(a, b-1)$ and the two twos correspond to the tiles with southeast vertex $(a+1,b-1)$ and $(a+1, b)$ respectively. On the other hand,  the corresponding segment of the continued fraction $\cf_{p/q}$ is $[\ldots, 2,2,1,1, \ldots]$. Thus in this case, $\delta=2$.

\begin{figure}
\begin{center}
\scalebox{1}{ 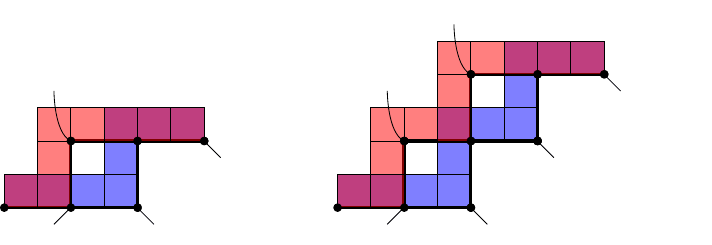} 
\\
\scalebox{1}{ 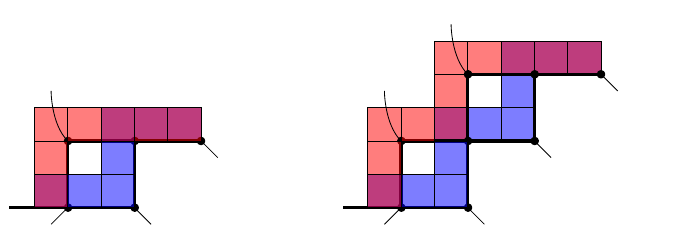}
 \caption{Proof of Lemma \ref{lem 5}}\label{fig 3}
\end{center}
\end{figure}

If the point $(a+1,b+1)$ lies between the lines then we have the situation illustrated in the right picture of Figure \ref{fig 3}, where the path of $(q,p)$ goes through the lattice point $(a+1,b+1)$  while the path of $(q+1,p)$ goes through the point $(a+2,b)$. By part (a), the two paths must meet again at the point $(a+2,b+1)$, after which there are again two cases. Either both paths go through the point $(a+3,b+1)$, in which case the corresponding segments of the continued fraction are $[\ldots 1,1,2,2,2,2,\ldots]$ and $[\ldots, 2,2,2,2,1,1,\ldots]$, and $\delta =2,2,2$;
or the point $(a+2, b+2)$ lies between the lines. Repeating the argument shows part (b).
\end{proof}

\subsection{Main result} 
 The fixed numerator conjecture holds as a result of the following theorem.
\begin{thm}\label{markovordering thm}
	Let $p$ and $q$ be positive integers such that $p<q$. Then 
$m_{p/q} < m_{p/(q+1)}$.
\end{thm}

\begin{proof} We can write the continued fraction of $m_{p/q}$ as a list of $q+p-1$ replaceable entries, $a_i = 1,1$ or 2. Whereas the continued fraction of $m_{p/(q+1)}$ 
 would have $q+p$ replaceable entries.
 We are considering replacements from the continued fraction of $m_{p/(q+1)}$ to the continued fraction of $m_{p/q}$. This means that the first replacement substitutes a pair $1,1$ by a $2$, and then the replacements alternate between replacing a $2$ with a pair $1,1$ and replacing a pair $1,1$ with a $2$ and so on.

  Moreover, Lemma~\ref{lem 5} implies that between two such replacements, both continued fractions consist entirely of 2s. Thus we can write
 \[m_{p/{q+1}}=N[\mu,1,1,\delta,2,\nu,2] \quad \textup{and} \quad  m_{p/{q}}=N[\mu,2,\delta,1,1,\nu'], \]
 where $\mu,\nu,\nu'$ are sequences of ones and twos, 
 and $\delta$ consists only of twos.
 
 First, we prove the result for an even number replacements by induction on the number of {\bf pairs} of replacements.
Suppose first that there is exactly  one pair of replacements, that is $\nu=\nu'$, so that $\nu$ and $\nu'$ contain no replacements. 

We would like to show $N[\mu, 1,1, \delta,2,\nu,2] - N[\mu, 2, \delta,1,1,\nu]> 0$.

\begin{eqnarray*}
\lefteqn{N[\mu, 1,1, \delta,2,\nu,2] - N[\mu, 2, \delta,1,1,\nu]}\\[8pt]
	&=& N[\mu, 1,1, \delta,2,\nu,1] +N[\mu, 1,1, \delta,2,\nu]- N[\mu, 2, \delta,1,1,\nu]\\[8pt]
	&=&N[\mu,1,1]N[\delta,2,\nu,1] + N[\mu,1]N[^-\delta,2,\nu,1]+ N[\mu,1,1]N[\delta,2,\nu] + N[\mu,1]N[^-\delta,2,\nu]\\
	&&-N[\mu,2]N[\delta,1,1,\nu] - N[\mu]N[^-\delta,1,1,\nu]\\[8pt]
	&=& N[\mu,1,1]N[\delta,2]N[\nu,1]+ N[\mu,1,1]N[\delta]N[^-\nu,1]\\
	&&+ N[\mu,1]N[^-\delta,2]N[\nu,1]+ N[\mu,1]N[^-\delta]N[^-\nu,1]\\
	&&+ N[\mu,1,1]N[\delta,2]N[\nu]+ N[\mu,1,1]N[\delta]N[^-\nu]\\
	&&+ N[\mu,1]N[^-\delta,2]N[\nu]+ N[\mu,1]N[^-\delta]N[^-\nu]\\
	&&- N[\mu,2]N[\delta,1,1]N[\nu] - N[\mu,2]N[\delta,1]N[^-\nu] \\
		&&- N[\mu]N[^-\delta,1,1]N[\nu] - N[\mu]N[^-\delta,1]N[^-\nu] 
\end{eqnarray*}
The first negative term plus the fifth term equals zero. The sum of the second negative term and the first term is positive. The sum of the third negative term and the third term is positive. The sum of the last negative term and the seventh term is positive. Therefore $N[\mu, 1,1, \delta,2,\nu,2] - N[\mu, 2, \delta,1,1,\nu]>0$ and this proves the base case of our induction.

\smallskip

Now we  proceed with the inductive step. 
Assume $N[\mu, 1,1, \delta,2,\nu,2] - N[\mu, 2, \delta,1,1,\nu'] >0$ when $\nu$ and $\nu'$ contain $n-1$ pairs of replacements. We would like to show $N[\mu, 1,1, \delta,2,\nu,2] - N[\mu, 2, \delta,1,1,\nu'] >0$ when $\nu$ and $\nu'$ contain $n$ pairs of replacements.
We have

\begin{eqnarray*}
\lefteqn{N[\mu, 1,1, \delta,2,\nu,2] - N[\mu, 2, \delta,1,1,\nu']}\\[8pt]
	&=& N[\mu, 1,1, \delta,2,\nu,1] +N[\mu, 1,1, \delta,2,\nu]- N[\mu, 2, \delta,1,1,\nu']\\[8pt]
	&=&N[\mu, 1,1, \delta,2]N[\nu,1] + N[\mu, 1,1, \delta]N[^-\nu,1] \\
	&&+ N[\mu, 1,1, \delta,2]N[\nu] + N[\mu, 1,1, \delta]N[^-\nu]\\
		&&- N[\mu, 2, \delta,1,1]N[\nu'] - N[\mu, 2, \delta,1]N[^-\nu']\\[8pt]
		&=& N[\mu, 1,1, \delta,2] ( N[\nu,1]+ N[\nu]  ) + N[\mu, 1,1, \delta] ( N[^-\nu,1]+ N[^-\nu] )\\
		&&-  N[\mu, 2, \delta,1,1]N[\nu'] - N[\mu, 2, \delta,1]N[^-\nu'].\\[8pt]
		&=& N[\mu, 1,1, \delta,2]  N[\nu,2] + N[\mu, 1,1, \delta]  N[^-\nu,2]\\
		&&-  N[\mu, 2, \delta,1,1]N[\nu'] - N[\mu, 2, \delta,1]N[^-\nu'].\\
	\end{eqnarray*}

Following Theorem \ref{lem replacedifference}, we may substitute $N[\mu,1,1,\delta,2] = N[\mu,2,\delta,2] + N[\mu^-]N[^-\delta,2]$, $ N[\mu, 1,1, \delta] =  N[\mu, 2, \delta] + N[\mu^-]N[^-\delta]$ and we can also substitute $N[\mu, 2, \delta,1] = N[\mu, 2, \delta]+ N[\mu, 2, \delta^-]$. Thus we have the following equivalent expressions.

\begin{eqnarray}
&&	N[\mu,2,\delta,2] N[\nu,2]  + N[\mu^-]N[^-\delta,2] N[\nu,2] \nonumber \\
&&+ 	N[\mu, 2, \delta] N[^-\nu,2] + N[\mu^-]N[^-\delta] N[^-\nu,2]\nonumber \\
&&-  N[\mu, 2, \delta,1,1]N[\nu'] - N[\mu, 2, \delta]N[^-\nu'] - N[\mu, 2, \delta^-]N[^-\nu']\nonumber \\[8pt]
&=& N[\mu,2,\delta,2] (N[\nu,2]- N[\nu']) + N[\mu, 2, \delta] ( N[^-\nu,2] -N[^-\nu'] )  \label{eq induction}\\
&&+ N[\mu^-]N[^-\delta,2] N[\nu,2]  + N[\mu^-]N[^-\delta] N[^-\nu,2]\nonumber
- N[\mu, 2, \delta^-]N[^-\nu']
	\end{eqnarray}
  Let us show that the first two terms in (\ref{eq induction}) are in fact positive. Indeed, the induction hypothesis says that $N[\mu,1,1,\zd,2,\nu,2]-N[\mu,2,\zd,1,1,\nu']>0 $ if $\nu,\nu'$ contain at most $n-1$ pairs of replacements. Observe that our $\nu$, $\nu'$ in (\ref{eq induction}) are of the form $\nu=[\mu_1,1,1,\zd_1,2,\nu_1],\nu'= [\mu_1,2,\zd_1,1,1,\nu_1']$, for some sequences $\mu_1,\zd_1,\nu_1,\nu_1'$. Since $\nu_1$ and $\nu'_1$ contain only $n-1$ pairs of replacements, the induction shows $N[\nu,2] > N[\nu']$ and $N[^-\nu,2] > N[^-\nu'] $.
 Using the same inequalities on the third and fourth term, we only need to show that the following expression is positive.

\begin{eqnarray*}
\lefteqn{N[\mu^-]N[^-\delta,2]N[\nu']+ N[\mu^-]N[^-\delta]N[^-\nu']- N[\mu, 2, \delta^-]N[^-\nu']}\\[8pt]
&=&N[\mu^-]N[^-\delta,2]N[\nu']+ N[\mu^-]N[^-\delta]N[^-\nu']\\
&&- N[\mu]N[ 2, \delta^-]N[^-\nu']-  N[\mu^-]N[ \delta^-]N[^-\nu']\\[8pt]
&=&N[\mu^-]N[^-\delta,2]N[\nu']- N[\mu]N[ 2, \delta^-]N[^-\nu']\\[8pt]
&=& N[^-\delta,2] (N[\mu^-]N[\nu'] - N[\mu]N[^-\nu'])
	\end{eqnarray*}
By Corollary \ref{cor 01}, this last expression is greater than zero. 

Next, we prove the result for an odd number of replacements. Let \[m_{p/{q+1}}=N[\mu,1,1,\nu,2] \quad \textup{and} \quad  m_{p/{q}}=N[\mu',2,\nu], \]
	where $\mu,\mu',$ and $\nu$ are sequences of ones and twos, $\mu$ and $\mu'$ start with a 2 and contain an even number of replacements. We would like to show that $N[\mu,1,1,\nu,2] - N[\mu',2,\nu]>0$.

\begin{eqnarray*}
	N[\mu,1,1,\nu,2] - N[\mu',2,\nu] &=& N[\mu,2,\nu,2] - N[\mu',2,\nu] +N[\mu^-]N[^-\nu,2]
	\end{eqnarray*}
The difference $N[\mu,2,\nu,2] - N[\mu',2,\nu] $ contains an even number of replacements and is therefore positive. Hence $N[\mu,1,1,\nu,2] - N[\mu',2,\nu]>0$. In any case, $m_{p/q} < m_{p/(q+1)}$.
	\end{proof}
\begin{remark}
 Note that if $p$ and $q$ are relatively prime, we can apply Theorem~\ref{markovordering thm} repeatedly to obtain the inequality in Conjecture \ref{constant num conjecture}.
\end{remark}

{}

 \end{document}